\documentclass[a4paper, 11pt]{amsart}
\usepackage{amsmath, amssymb, amsthm, mathrsfs}
\usepackage[mathscr]{eucal}
\usepackage[all]{xy}
\usepackage[dvipdfmx]{graphicx}
\usepackage{comment}
\usepackage{mathtools}
\usepackage{cleveref}
\usepackage[backend=biber, style=math-numeric, sorting=nyt, maxnames=100]{biblatex}
\theoremstyle{plain}
 \newtheorem{thm}{Theorem}[section]
 \newtheorem{lemma}[thm]{Lemma}
 
 \newtheorem{prop}[thm]{Proposition}
 \newtheorem{claim}[thm]{Claim}
\theoremstyle{definition}
  \newtheorem{definition}[thm]{Definition}

\theoremstyle{remark}
  \newtheorem{rem}[thm]{Remark}
  \newtheorem{example}[thm]{Example}

  \makeatletter
  \@addtoreset{equation}{section}
  \makeatother

\newcommand{\abs}[1]{\lvert#1\rvert}
\newcommand{\acts}{\curvearrowright}
\newcommand{\G}{\mathcal{G}}
\newcommand{\h}{\mathcal{H}}
\newcommand{\K}{\mathcal{K}}

\begin{document}
\title{Cost of Inner Amenable Groupoids}

\author{Robin Tucker-Drob}
\author{Konrad Wr\'{o}bel}

\begin{abstract}
    Kida and Tucker-Drob recently extended the notion of inner amenability from countable groups to discrete p.m.p. groupoids. In this article, we show that inner amenable groupoids have ``fixed priced 1'' in the sense that every principal extension of an inner amenable groupoid has cost 1. This simultaneously generalizes and unifies two well known results on cost from the literature, namely, (1) a theorem of Kechris stating that every ergodic p.m.p.\ equivalence relation admitting a nontrivial asymptotically central sequence in its full group has cost 1, and (2) a theorem of Tucker-Drob stating that inner amenable groups have fixed price 1.
    \end{abstract}

\maketitle

\section{Introduction}
Cost is an $[0,\infty]$-valued invariant of p.m.p. orbit equivalence relations that was first introduced by Levitt\cite{Le95} and significantly developed by Gaboriau\cite{Ga98}\cite{Ga00}\cite{Ga02OE}. By work of Connes-Feldman-Weiss\cite{CFW81}, every aperiodic amenable equivalence relation has cost $1$. Kechris showed in (\cite{Kec10}, Theorem 8.1) that the existence of a nontrival asymptotically central sequence in the full group of an ergodic p.m.p.\ equivalence relation implies the equivalence relation has cost 1. More exposition on the many results of cost theory can be found in the surveys of Kechris and Miller\cite{KeMi04} and Furman\cite{Fu11}.

A discrete p.m.p.\ groupoid $\G$ is said to have fixed price if every principal extension of $\G$ has the same cost. Gaboriau proved that free groups have fixed price and asked whether every group has fixed price. Since then, fixed price has been shown for several large classes of groups, including finite, infinite amenable\cite{OW80}, strongly treeable\cite{Ga00}, and inner amenable groups\cite{TD14} amongst many others. Recently, Hutchcroft and Pete showed that Kazhdan groups have a principal extension with cost 1\cite{HuPe18}, but it is an open question whether these groups have fixed price.

Inner amenable groups were first introduced by Effros\cite{Eff75} in relation with property Gamma of a von Neumann algebra. Examples of inner amenable groups include infinite amenable groups, groups with infinite center, and groups admitting a ergodic p.m.p.\ action which is stable in the sense of Jones and Schmidt \cite{JS}. Kerr and Tucker-Drob have shown that dynamical alternating groups associated to topologically free actions of amenable groups on the Cantor set are inner amenable, and they use this to exhibit uncountably many pairwise nonisomorphic, finitely generated simple nonamenable inner amenable groups\cite{KeTD19}. Recently, Kida and Tucker-Drob defined inner amenability for discrete p.m.p.\ groupoids\cite{KiTD18} and showed that the action groupoid associated to a compact p.m.p.\ action of an inner amenable group is inner amenable. They also note that not every free p.m.p. action of an inner amenable group gives an inner amenable action groupoid. In particular, the action groupoid associated to the Bernoulli shift of any nonamenable group is not inner amenable\cite[Corollary 6.3]{KiTD18}.

\begin{thm}\label{thm:cost1}
Assume $\varphi: R\rightarrow \G$ is a principal groupoid extension of an inner amenable groupoid $\G$. Then $C_{\mu^0_R}(R)=1$.
\end{thm}

Specializing this to the case of countable groups recovers Tucker-Drob's result that inner amenable groups have fixed price 1, and specializing to the case of equivalence relations recovers Kechris's theorem that equivalence relations with a nontrivial asymptotically central sequence in their full group have cost 1 (since these equivalence relations are shown to be inner amenable in \cite{KiTD18}).

In order to prove Theorem \ref{thm:cost1}, we make heavy use of a generalization of von Neumann's notion of amenable actions, from the setting of groups to the setting of groupoids. We obtain the following key structural result along the way, generalizing a result from \cite{TD14} which only applied to groups.

\begin{thm}
If $\G$ is an inner amenable groupoid, and $\h \leq \G$ is a nowhere amenable subgroupoid, then there is a groupoid $\K$ such that $\h$ is q-normal in $\K$ and $\K$ is q-normal in $\G$.
\end{thm}

\section{Preliminaries}
\subsection{Groupoids}
\begin{definition}
A \textbf{groupoid} $\G$ is a small category in which every morphism is an isomorphism. We refer to the set of objects as the unit space, written as $\G^0$. There are source $s_\G:\G\rightarrow \G^0$ and range $r_\G:\G \rightarrow \G^0$ maps that send an element of the groupoid to its source and range, respectively, and an inclusion map $i_\G : \G ^0 \rightarrow \G$, that sends a unit to the identity morphism at that unit. When there is no confusion, we will drop the subscripts on $s$, $r$, and $i$, and we will identify $\G ^0$ with its image in $\G$ under $i$.

We say a groupoid $\G$ is \textbf{principal} if the map $g\mapsto (r(g),s(g))$ is injective.
\end{definition}

An equivalence relation $R$ on a set $X$ is naturally a principal groupoid, with unit space $X$, source and range maps the right and left projections respectively, and composition given by the rule $(z,y)(y,x)=(z,x)$. Moreover, each principal groupoid is naturally isomorphic as a groupoid to an equivalence relation via the map $g\mapsto (r(g),s(g))$. Given this transparent equivalence of categories, we will freely and frequently identify principal groupoids with their associated equivalence relation.
%

\begin{definition}
A \textbf{discrete Borel groupoid} is a groupoid where both $\G$ and $\G^0$ are standard Borel spaces, the source, range, and inclusion maps $s$, $r$, and $i$ are all Borel, $s$ and $r$ are countable-to-one, and the multiplication and inverse maps are Borel.
\end{definition}

\begin{definition}
A \textbf{discrete p.m.p. groupoid} is a pair $(\G,\mu^0_\G)$ where $\G$ is a discrete Borel groupoid and $\mu^0_\G$ is a Borel probability measure on $\G^0$ satisfying $\int_{\G^0} c_x^r d\mu^0_\G=\int_{\G^0} c_x^s d\mu^0_\G$ where $c_x^r$ and $c_x^s$ refer to the counting measure on $r^{-1}(x)$ and $s^{-1}(x)$ respectively. Set $\mu^1_\G:=\int_{\G^0} c_x^r d\mu^0_\G=\int_{\G^0} c_x^s d\mu^0_\G$ to be this $\sigma$-finite measure on $\G$.
\end{definition}

Again, we will drop the subscript on the measures when there is no cause for confusion.

\begin{definition}
A subset $A\subseteq \G^0$ is said to be $\G$-\textbf{invariant} if $\mu^0(\G\cdot A \triangle A)=0$ where $\G\cdot A=\{x\in \G^0\: |\:\exists g\in \G\text{ with }s(g)\in A \text{ and }r(g)=x\}$.
\end{definition}

\begin{definition}
A discrete p.m.p. groupoid $\G$ is called \textbf{ergodic} if every $\G$-invariant subset $A\subseteq \G^0$ is $\mu^0$-null or conull.
\end{definition}

\begin{example}
Let $G\acts (X,\mu)$ be a p.m.p. action of a countable group on a standard probability space $(X,\mu )$. We define the discrete p.m.p. groupoid $\G = G\ltimes X$ with underlying set $G\times X$ and unit space $\G^0=X$, with the groupoid operation $(g,h\cdot x)(h,x):=(gh,x)$. A groupoid that arises through such a process is called an \textbf{action groupoid}. If this action is ergodic, so is the groupoid it generates.
\end{example}

\begin{definition}
An \textbf{extension} of a discrete p.m.p.\ groupoid $\G$ is a discrete p.m.p.\ groupoid $\h$ together with a measure preserving groupoid homomorphism $\phi: (\h , \mu _{\h}^1) \rightarrow (\G ,\mu _{\G}^1)$. We call this a \textbf{principal extension} of $\G$ if the groupoid $\h$ is principal.
\end{definition}

As explained in \cite{BTD}, the category of extensions of $\G$ is equivalent to the category of p.m.p. actions of $\G$.

\begin{example}
Let $G\acts (X,\mu)$ be a free p.m.p. action. Then the map $\phi:G\ltimes (X,\mu)\rightarrow G$ defined by $(g,x)\mapsto g$ is a principal groupoid extension.
\end{example}
More detail about groupoid extensions can be found in \cite{BTD} and \cite{KiTD18}.

\begin{definition}
A \textbf{measurable bisection} of a discrete p.m.p. groupoid $\G$ is a Borel subset $\sigma$ of $\G$ such that the restrictions $r|_\sigma$ and $s|_\sigma$ are each bijections of $\sigma$ with a conull subset of $\G^0$. A subset $\sigma$ of $\G$ is called a \textbf{partial measurable bisection} if $r|_{\sigma},s|_\sigma$ are only assumed to be injections.
\end{definition}

\begin{definition}
The \textbf{full group} of a discrete p.m.p. groupoid $\G$ is the set, denoted by $[\G]$, of all measurable bisections. The \textbf{pseudogroup} of $\G$ is the set, denoted by $[[\G]]$, of all partial bisections. We identify two partial bisections $\sigma_1$ and $\sigma_2$ if their symmetric difference $\sigma_1\triangle\sigma_2$ is $\mu^1$-null. The full group admits a complete separable metric, namely $d(\sigma_1,\sigma_2):=\mu^1(\sigma_1\triangle\sigma_2)$.
\end{definition}

For subsets $A,B\subseteq \G$ define $A^{-1}:=\{ g^{-1} | g\in A \}$ and $AB := \{ gh | g\in A,\ h\in B,\text{ and }s(g)=r(h)\}$. The full group and the full pseudogroup of $\G$ are then a group and inverse semigroup respectively, under these operations. For $g\in \G$ and $A\subset \G$ we also define $gA := \{ g\} A$.

For a groupoid $\G$, and subset $A\subseteq \G^0$, we let $\G_A:=\{g\in \G \: \vert \: r(g),s(g)\in A\}$.
Fix a partial measurable bisection $\sigma \in [[\G]]$. For $g\in \G_{r(\sigma)}$, define the conjugate of $g$ by $\sigma$, denoted $g^{\sigma}$, to be the unique element of $\sigma^{-1} g\sigma$. Likewise, for $D\subseteq \G$, define $D^\sigma=\{ g^\sigma\:|\: g\in D \}$. For a function $f: \G \rightarrow \mathbb{C}$, define $f^{\sigma}: \G \rightarrow \mathbb{C}$ by

\[ f^\sigma(g)=
\begin{cases}
f(g^{\sigma^{-1}}) & \text{ if } g \in \G_{s(\sigma)} \\
0 & \text{ otherwise}
\end{cases}
\]

\subsection{Actions of Groupoids and Amenability}

\begin{definition}
A \textbf{locally countable fibered space} over a standard measure space $(X,\mu)$ consists of a standard Borel space $W$ along with a countable-to-one Borel map $p: W\rightarrow X$. 
For $A\subseteq X$ we define $W^A:=p^{-1}(A)$ and set $W^x:=W^{\{x\}}$. We also define $\nu(A):=\int_X |W^x\cap A|d\mu$.

If $\G$ is a discrete p.m.p.\ groupoid and $W$ is a locally countable fibered space over $\G ^0$, then we define $\G \ast W:=\{(g,w)\in \G\times W \: |\: s(g)=p(w)\}$.
\end{definition}

Note that a discrete p.m.p. groupoid, together with either its source or range map, is a locally countable fibered space over $\G^0$, with $\nu_r=\nu_s=\mu^1$.

\begin{definition}
A \textbf{(left) Borel action of a p.m.p. groupoid $\G$ on a locally countable fibered space} $p: W\rightarrow \G^0$ is a Borel map $\alpha : \G\ast W \rightarrow W$, such that
\begin{itemize}
\item[(1)] $\alpha (g,w)\in W^{r(g)}$ for each $g\in \G$ and $w\in W^{s(g)}$,
\item[(2)] for each $g\in \G$ the map $\alpha _g : w\mapsto \alpha (g,w)$, is a bijection from $W^{s(g)}$ to $W^{r(g)}$, and
\item[(3)] $\alpha _g\alpha _h = \alpha _{gh}$ whenever $s(g)=r(h)$
\end{itemize}
where we denote $\G\ast W = \{ (g,w) | s(g)=p(w) \}$.

We will also simply write $gw$ for $\alpha (g,w)$. For subsets $A\subseteq \G$ and $V\subseteq W$, denote $AV:=\left\{ gw\:| \: g\in A \text{ and } w\in V\right\}$. For $g\in \G$ and $V\subseteq W$, we define $gV := \{ g \} V$.
\end{definition}

\begin{example}
The \textbf{left translation action} $\lambda$, of a groupoid $\G$ on itself is defined, for $g,h\in \G\ast\G = \{ (g,h) | s(g)=r(h) \}$, by $\lambda (g,h):=gh$.
\end{example}


\begin{definition}
A \textbf{measurable section} of a locally countable fibered space $p:W\rightarrow (X,\mu)$ is a Borel subset $\sigma$ of $W$ such that the restriction $p|_\sigma$, of $p$ to $\sigma$, is a bijection of $\sigma$ with a conull subset of $X$. A subset $\sigma$ of $W$ is called a \textbf{partial measurable section} if $p|_{\sigma}$ is only assumed to be injective.
\end{definition}

Suppose we have an action of a discrete p.m.p.\ groupoid $\G$ on a locally countable fibered space $p: W\rightarrow \G^0$. Let $g\in \G$ and let $\sigma \subseteq W$ be a partial measurable section. We say that $g$ fixes $\sigma$ if $\varnothing \neq g\sigma\subseteq \sigma$. Notice that, in this case, the set $g\sigma$ contains a single point, so we will abuse notation and use $g\sigma$ to also denote this point. 

\begin{definition}
The \textbf{stabilizer} of a partial section $\sigma$ is defined to be the set
\[
\G_\sigma:=\{g\in\G \: | \: g \text{ fixes } \sigma\}=\{g\:|\:\varnothing\neq g\sigma\subseteq\sigma\}.
\]
\end{definition}

\begin{definition}
Let $\G \acts W$ be a Borel action of a discrete p.m.p. groupoid $\G$ on a locally countable fibered space $W$ over $\G ^0$. The action is called \textbf{amenable} if there exists a sequence of Borel functions $(f_n)_{n\in \mathbb{N}}:W\rightarrow [0,1]$ such that

\begin{enumerate}
   \item $\sum\limits_{w\in W^x} f_n(w)=1$ for $\mu^0$-almost every $x\in\G^0$ and all $n\in \mathbb{N}$ \label{amen1} 
   \item $\sum \limits_{w\in W^{r(g)}} |f_n (w) - f_n (g^{-1}w)| \rightarrow 0$ as $n\rightarrow \infty$ for $\mu^1$-almost every $g\in\G$. \label{amen2} 
\end{enumerate}
\end{definition}

Notice that this generalizes von Neumann's notion of amenable action when $\G$ is actually a group.

\begin{rem} \label{invrmk}
If a sequence $f_n$ is $\G$-asymptotically invariant as in item (2), then, by the bounded convergence theorem, it is asymptotically invariant under the full group, i.e. $\|f_n-\sigma f_n\|_1\rightarrow 0$ for $\sigma\in[\G]$. Conversely, if a sequence is asymptotically invariant under the full group, then a subsequence is $\G$-asymptotically invariant as in item (2).
\end{rem}

\begin{definition}
A \textbf{mean} on $(W,\nu)$ is a norm one positive linear functional on $L^\infty(W,\nu)$. 
\end{definition}

We'll often treat means as finitely additive probability measures by letting $m(A):=m(1_A)$ for $A\subseteq W$ a $\nu$-measurable subset.

\begin{definition}
Let $m\in (L^\infty(W,\nu))^*$ be a mean on a fibered space $p:W\rightarrow \G^0$. We say that $m$ is \textbf{equidistributed} if for every measurable set $A\subseteq \G^0$,
\[m(W^A)=\mu^0(A).\]
\end{definition}

\begin{prop}\label{eqamen}
Let $\G$ be a discrete p.m.p. groupoid that acts on the fibered space $p:W\rightarrow \G^0$. Assume the action admits an equidistributed mean $m\in (L^\infty(W,\nu))^*$ such that $m(\sigma A)=m(A)$ for every measurable set $A$ and $\sigma \in [\G]$. Then the action $\G\acts W$ is amenable.
\end{prop}

\begin{proof}

Associate to every non-negative unit vector $f\in L^1(W,\nu)$ a function $p_f\in L^1(\G^0,\mu^0)$ defined by 
\[p_f(x):=\sum_{w\in W^x}f(w).\]
Since $m$ is equidistributed, given a net $(f_i)$ converging to $m$, the net $(p_{f_i})$ weak-* converges to the function $1$.

The proof of the following claim follows the proof given in \cite[Claim 3.19]{KiTD18}.

\begin{claim}
Let $f\in L^1(W,\nu)$ be a non-negative function with $\|f\|_1=1$. Then there exists a non-negative $g\in L^1(W,\nu)$ which satisfies $p_g=1_{\G^0}$ and $\|f-g\|_1 = \|p_f-1_{\G^0}\|_1$.
\end{claim}

\begin{proof}[Proof of Claim]
Let $f_0 := f$. We proceed by transfinite induction on countable ordinals $\alpha$ to define
a non-negative function $f_\alpha\in L^1(W,\nu)$ satisfying, for all $\beta<\alpha$:
\begin{enumerate}
    \item $\|f_\beta -f_\alpha\|_1 = \|p_{f_\alpha}-p_{f_\beta}\|_1$
    \item For almost every $x\in\G^0$, if $p_{f_\beta}(x)\leq 1$, then $p_{f_\beta}(x)\leq p_{f_\alpha}(x)\leq 1$
    \item For almost every $x\in\G^0$, if $p_{f_\beta}(x)\geq 1$, then $p_{f_\beta}(x)\geq p_{f_\alpha}(x)\geq 1$
    \item If $\|p_{f_\beta}- 1_{\G^0}\|_1 > 0$, then $\|p_{f_\alpha} - 1_{\G^0} \|_1 < \|p_{f_\beta} -1_{\G^0}\|_1$, and if $p_{f_\beta} = 1_{\G^0}$, then $f_\alpha=f_\beta$.
\end{enumerate}

If $\alpha$ is a limit ordinal, take an increasing sequence $\beta_1<\beta_2<\dots$ such that $\alpha=\sup_i \beta_i$. The sequence $(p_{f_{\beta_i}})_i\subseteq L^1(\G^0,\mu^0)$ is Cauchy by properties (2) and (3). By property (1), the sequence $(f_{\beta_i})_i\subseteq L^1(W,\nu)$ is Cauchy. Let $f_\alpha$ be the limit point of this sequence. 

If $\alpha$ is a successor ordinal and $p_{f_{\alpha-1}}=1_{\G^0}$, then let $f_\alpha:=f_{\alpha-1}$. 

If $\alpha$ is a successor ordinal and $p_{f_{\alpha-1}}\not=1_{\G^0}$, then there exists $\varepsilon>0$ such that the sets $A_0:=\{x| p_{f_{\alpha-1}}(x)<1-\varepsilon\}$ and $A_1:=\{x| p_{f_{\alpha-1}}(x)>1+\varepsilon\}$ both have positive measure. For $i\in \{0,1\}$, find partial measurable sections $C_i\subseteq W^{A_i}$ with $\nu(C_0)=\nu(C_1)>0$ and $\varepsilon':=\inf\{f_{\alpha-1}|_{C_1}\}>0$. By letting $\delta:=\min\{\varepsilon, \varepsilon'\}$ and $f_\alpha:=f_{\alpha-1}+\delta(1_{C_0}-1_{C_1})$, we get a function with the required properties.

By property (4), there exists a countable ordinal $\kappa$ such that $p_{f_\kappa}=1_{\G^0}$, and setting $g:=f_\kappa$ finishes the proof of the claim.
\end{proof}

Now, for a fixed finite collection $\Delta\subseteq [\G]$ and $\varepsilon>0$, we consider the convex set $\{(f-\delta f)_{\delta\in \Delta}\times (p_f-1_{\G^0})\:|\: \|f\|_1=1, f\geq 0\}\subset L^1(W,\nu)^\Delta\times L^1(\G^0,\mu^0)$. By the Hahn-Banach theorem, this set contains $0$ in its weak closure. So, by Mazur's Theorem, $0$ is in the norm closure of our set.

By the claim, for every finite collection of sections $\Delta \subseteq [\G]$ and $\varepsilon>0$, there exists a positive norm one function $f$ such that $p_f=1_{\G^0}$ and $\max_{\delta \in \Delta}\|f-\delta f\|_{1}\leq \varepsilon$. Fix a countable dense subset $(\delta_k)\subseteq [\G]$. Take $f_n$ as above satisfying $\max_{1\leq k\leq n}\|f_n-\delta_k f_n\|_{1}\leq \frac{1}{n}$ and $p_{f_n}=1_{\G^0}$. The functions $f_n$ satisfy item (1) in the definition of amenability because $p_{f_n}=1_{\G^0}$. 

Let $\delta\in[\G]$ and $\varepsilon>0$. There exists $I$ such that $\mu^1(\delta\triangle \delta_I)<\frac{\varepsilon}{4}$ since $(\delta_k)$ is dense. Let $N>\max(\frac{2}{\varepsilon},I)$. For $n>N$,
\[
\|f_n-\delta f_n\|_{1} \leq \|f_n-\delta_I f_n\|_{1} + 2\mu^1(\delta\triangle \delta_I)<\frac{\varepsilon}{2}+2\frac{\varepsilon}{4}<\varepsilon.
\]
By ~\cref{invrmk}, a subsequence of $(f_n)$ satisfies item (2) in the definition of amenability and hence the action $\G\acts W$ is amenable.
\end{proof}


\begin{definition}
A discrete p.m.p. groupoid $\G$ is called \textbf{amenable} if the left translation action of $\G$ on itself is amenable.
\end{definition}

In the case when $\G$ is an equivalence relation, this corresponds with the notion of amenability in the category of equivalence relations\cite{KeMi04}. A study of amenable groupoids can be found in \cite{AnRe00}.

\begin{definition}
A groupoid $\G$ is called \textbf{nowhere amenable} if for every positive measure subset $A\subseteq \G^0$, the groupoid $\G_{A}=\{ g\in \G \:|\: s(g),r(g)\in A \}$ is nonamenable.
\end{definition}

In \cite{KiTD18}, Kida and Tucker-Drob introduced the following generalization of inner amenable groups.

\begin{definition}\label{iadef}
A discrete p.m.p. groupoid $\G$ is called \textbf{inner amenable} if there exists a mean $m\in (L^\infty(\G,\mu^1))^*$ such that
\begin{enumerate}
   \item $m(\G_{A})=\mu^0(A)$ for every $\mu^0$-measurable $A\subseteq \G^0$\label{iamean1}
   \item $m(A^\sigma)=m(A)$ for every $\mu^0$-measurable $A\subseteq \G^0$ for every $\sigma \in [\G]$\label{iamean2}
   \item $m(D)=0$ for every $\mu^1$-measurable $D\subseteq \G$ with $\mu^1(D)<\infty$\label{iamean3}
   \item $m(A)=m(A^{-1})$ for every $\mu^0$-measurable $A\subseteq \G^0$ \label{iasymm}
\end{enumerate}
\end{definition}

\subsection{q-Normality and Cost}

\begin{definition}
Fix a discrete p.m.p. groupoid $\G$.
A subset $A\subseteq \G$ is said to \textbf{generate} $\G$ if the union $\left \langle A\right \rangle:=\bigcup_{n\in\:\mathbb{N}} (A\cup A^{-1})^n$ is a $\mu^1$-conull subset of $\G$.
\end{definition}

\begin{definition}
A subset of a discrete p.m.p.\ groupoid $A\subseteq\G$ is called \textbf{aperiodic} if for almost every $x\in \G^0$, the set $s^{-1}(x)\cap A$ is infinite.
\end{definition}

Sorin Popa introduced the notion of q-normality in \cite{Po06SomeComp}.

\begin{definition}
A subgroupoid $\h \leq \G$ is \textbf{q-normal} in $\G$ if there exists a countable collection of partial sections $\Sigma \subset [[\G]]$ generating $\G$ such that for every $\sigma\in\Sigma$, the set $\h^\sigma \cap \h$ is aperiodic on $s(\sigma)$.
\end{definition}

\begin{prop}\label{qnormext}
If the groupoid $\h$ is q-normal in $\G$ and $\varphi: \K\rightarrow \G$ is a groupoid extension of $\G$, then $\varphi^{-1}(\h)$ is q-normal in $\K$.
\end{prop}

\begin{proof}
Let $\Sigma$ be a countable collection of partial sections of $\G$ witnessing the q-normality of $\h$ in $\G$. Let $\Sigma':=\{\varphi^{-1}(\sigma)\:|\: \sigma \in \Sigma\}$. Notice $\Sigma'$ is a collection of sections of $\K$ that generate $\K$. Now $\varphi^{-1}(\sigma)\varphi^{-1}(\h)\varphi^{-1}(\sigma)^{-1}\cap \varphi^{-1}(\h)\supseteq \varphi^{-1}(\sigma\h\sigma^{-1}\cap \h)$. This along with the fact that groupoid extensions are surjective lets us check q-normality of $\varphi^{-1}(\h)$ in $\K$.
\end{proof}

\begin{prop}
Let $\varphi: \h \rightarrow \G$ be an extension of an amenable groupoid $\G$. Then $\h$ is amenable.
\end{prop}

\begin{proof}
If $(f_n)$ is a sequence witnessing the amenability of $\G$, then the sequence $(f_n\circ \varphi )$ witnesses the amenability of $\h$.
\end{proof}

\begin{definition}
The \textbf{cost} of a principal discrete p.m.p. groupoid $R$ is defined
\[C_{\mu^0}(R)=\inf_{A\text{ generating }R} \mu^1(A).\]
\end{definition}

\begin{prop}\label{qnorm}
Let $S\leq R$ be principal discrete p.m.p. groupoids. If $S$ is q-normal in $R$ then $C_{\mu^0_R}(R)\leq  C_{\mu_R^0}(S)$.
\end{prop}

This is proved in Proposition A.2 in \cite{TD14} and was first observed by Furman in \cite{Ga00}.





\section{A Folklore Lemma in the Groupoid Setting}

\begin{definition}
An action of a discrete countable group $G$ on a space $X$ is called \textbf{amenable} if there exists a finitely additive probability measure \(m:X\rightarrow [0,1]\) that is invariant under the group action.
\end{definition}

The following folklore lemma's origins go back to von Neumann's original paper which introduced the notion of amenability \cite{vN29}.

\begin{lemma}
Let a discrete countable group $G$ act on $X$ and $Y$. Suppose the action \(G \acts X\) is amenable. If the action of the stabilizer $G_x \acts Y$ is an amenable action for every $x\in X$, then \(G\acts Y\) is amenable.
\end{lemma}

We generalize this lemma in the following manner. 

\begin{lemma}\label{lm1}
Let $\G$ be a discrete p.m.p. groupoid. Fix actions of $\G$ on the locally countable Borel fibered spaces $p:W\rightarrow \G^0$ and $q:V\rightarrow \G^0$. Suppose $\G \acts W$ is an amenable action.  Suppose we have a countable collection of nontrivial partial measurable sections $\Sigma=\{\sigma_i\}_{i\in\mathbb{N}}$ of $W$ with the following properties
\begin{itemize}
\item $(\G\sigma_i)\cap (\G\sigma_j)=\varnothing$ if $i\neq j$
\item $\bigsqcup_{i\in \mathbb{N}} (\G\sigma_i)=W$
\item the restricted action $\G_{\sigma_i} \acts q^{-1}(p(\sigma_i))$ is amenable for all $i\in \mathbb{N}$.
\end{itemize}

Then the action $\G\acts V$ is amenable.
\end{lemma}

\begin{proof}
Fix $(f_n)$ an amenability sequence for the action $\G\acts W$ and denote $f_n |_{W^x}$ by $f_n^x$.

Pick a sequence of Borel subsets $(W_n)_{n\in \mathbb{N}}\subseteq W$ such that for every $n\in\mathbb{N}$ and for every $x\in \G^0$
\begin{itemize}
\item $|W_n\cap W^x|$ is finite
\item $f_n^x(W_n \cap W^x)>1-2^{-n}$.
\end{itemize}
We define
\[\hat f_n(w)= \begin{cases}
      \frac{f_n(w)}{f_n(W_n \cap W^{p(w)})} & w\in W_n\\
      0 & \text{otherwise}
   \end{cases}
\]
by restricting supports and renormalizing. The sequence $(\hat f_n)$ also witnesses the amenability of the action $\G\acts W$ and has finite support on each fiber. By replacing $f_n$ by $\hat f_n$, we may therefore assume that each of the functions $f_n$ is supported on $W_n$.

For $w\in W$, we denote by $\sigma_w$ the unique element $\sigma$ of $\Sigma$ such that $w\in \G\sigma$.  Now, we would like to find a measurable function $\phi:W\rightarrow \G$ such that $r(\phi(w))=p(w)$ and $\{w\}=\phi (w)\sigma_w$. Consider the map $\phi': \G \times \Sigma\rightarrow W$ defined by $(g,\sigma) \mapsto g\sigma$. This is surjective by the hypothesis on $\Sigma$ and countable-to-one since $\G$ has countable fibers. By the Lusin-Novikov Uniformization Theorem\cite{Ke95}, there is an injective Borel map $\phi^*: W\rightarrow \G \times \Sigma$ with $\phi'(\phi^*(w))=w$. By composing with a projection to $\G$, this is measurable and we get the map $\phi$ we were looking for. By abuse of notation we identify $\phi (w)\sigma_w$ with the point $w$ it contains.

For $g\in \G$ and $w\in W^{s(g)}$, we let $h_{g,w}:=\phi(gw)^{-1}g\phi(w)$. Notice that $h_{g,w}\in \G_{\sigma_w}$ since
\begin{align*}
h_{g,w}\sigma_w=\phi(gw)^{-1}g\phi(w)\sigma_w&=\phi(gw)^{-1}(gw) =\phi(gw)^{-1}(\phi(gw)\sigma_{gw}) \\
&= id_{s(\phi(gw))}\sigma_{gw}\in \sigma_{gw}=\sigma_w
\end{align*}
Let $(D_n)_{n\in \mathbb{N}}$ be an increasing sequence of finite $\mu^1$-measure subsets of $\G$ which exhaust the space. For each $\sigma\in \Sigma$, we pick a sequence $a_{\sigma,n}$ witnessing the amenability of the action $\G_\sigma\acts q^{-1}(p(\sigma))$. We may choose these sequences in such a way that there exists a sequence $(D_n')_{n\in\mathbb{N}}$ satisfying the following
\begin{itemize}
\item $\dfrac{\mu^1(D_n'\cap D_i)}{\mu^1(D_i)}>1-2^{-n}$  for all $i\leq n$
\item for every $ g$ in $ D_n'$ and for every $w$ in $W^{s(g)}\cap W_n $
\begin{equation}
\left \Vert a_{\sigma_w,n}^{r(h_{g,w})}-h_{g,w} a_{\sigma_w,n}^{s(h_{g,w})}\right \Vert_{\ell^1(V^{r(h_{g,w})})} <2^{-n} \label{restriction}
\end{equation}
\end{itemize}
This is accomplished as follows. First take an amenability sequence $\hat a_{\sigma,n}$ for the action $\G_\sigma\acts q^{-1}(p(\sigma))$. Here, for a given $g\in \G$ and $n\in \mathbb{N}$, the set $W^{s(g)}\cap W_n$ is finite.

Fix an element $g\in \G$. Let $N(n,g)$ be the least integer such that for all $N\geq N(n,g)$,
\[
\left \Vert \hat a^{r(h_{g,w})}_{\sigma_w,N} - h_{g,w}\hat a^{s(h_{g,w})}_{\sigma_w,N}\right \Vert_{\ell^1(V^{r(h_{g,w})})} \leq 2^{-n} \text{ for every } w\in  W^{s(g)}\cap W_n.\]
Such an $N(n,g)$ exists because $h_{g,w}\in \G_{\sigma_w}$ and by definition of $\hat a_{\sigma,n}$ being an amenability sequence. For each $c\in \mathbb{N}$, define $\hat D_n(c):=\left \{g\in D_n \:|\: N(n,g)<c \right \}$. The sets $\hat D_n(c)$ increase to $D_n$ as $c\rightarrow \infty$. Fix $c_n$ such that $\frac{\mu^1(\hat D_n(c_n)\cap D_i)}{\mu^1(D_i)}>1-2^{-n}$ for every $i\leq n$. Define $a_{\sigma,n}:=\hat a_{\sigma,c_n}$ and $D_n':=\hat D_n(c_n)$.

Now, for each $n\in\mathbb{N}$, define $\xi_n(v):=\sum_{w\in W^{q(v)}} a_{\sigma_w,n}(\phi(w)^{-1}v) f^{q(v)}_n(w)$. This is defined since $\phi(w)^{-1}v\in q^{-1}(p(\sigma_w))$. We show the sequence $(\xi_n)_{n\in\mathbb{N}}$ witnesses the amenability of the action $\G \acts V$. We first check it satisfies item (2) in the definition of amenability. Let $g\in \G$ and let $x=s(g)$ and $y=r(g)$.
\begin{align*}
\Vert \xi_n^y-g \xi_n^x\Vert_{\ell^1(V^y)} &=\sum_{v\in V^y}\left | \sum\limits_{w\in W^{y}} a_{\sigma_w,n}(\phi(w)^{-1} v) f^{y}_n(w) -\sum\limits_{w\in W^{x}} a_{\sigma_w,n}(\phi(w)^{-1}(g^{-1}v)) f^{x}_n(w) \right |
\end{align*}
\begin{align}
&\hspace{.65in}\leq \sum_{v\in V^y}\left | \sum\limits_{w\in W^{y}} a_{\sigma_w,n}(\phi(w)^{-1}v) f^{y}_n(w) -\sum\limits_{w\in W^{x}} a_{\sigma_w,n}(\phi(g w)^{-1}v) f^{x}_n(w) \right | \label{eq1}\\
&\hspace{.65in} +\sum_{v\in V^y}\left | \sum\limits_{w\in W^{x}} a_{\sigma_w,n}(\phi(g w)^{-1} v) f^{x}_n(w) -\sum\limits_{w\in W^{x}} a_{\sigma_w,n}(\phi(w)^{-1}g^{-1}v) f^{x}_n(w) \right | \label{eq2}
\end{align}

Let's first look at \cref{eq1} now and bound it by rewriting it as follows.
\begin{align*}
&\sum_{v\in V^y}\left | \sum\limits_{w\in W^{y}} a_{\sigma_w,n}(\phi(w)^{-1} v) f^{y}_n(w) -\sum\limits_{w\in W^{x}} a_{\sigma_w,n}(\phi(g w)^{-1} v) f^{x}_n(w) \right |\\
&\hspace{.15in}= \sum_{v\in V^y}\left | \sum\limits_{w\in W^{y}} a_{\sigma_w,n}(\phi(w)^{-1} v) f^{y}_n(w) -\sum\limits_{w\in W^{y}} a_{\sigma_w,n}(\phi(w)^{-1} v) f^{x}_n(g^{-1} w) \right |\\
&\hspace{.15in}\leq \sum\limits_{v\in V^y} \sum\limits_{w\in W^y} \left \vert a_{\sigma_w,n}(\phi(w)^{-1} v) \right \vert \left \vert f^y_n(w)-f^x_n(g^{-1} w)\right \vert \\
&\hspace{.15in}= \sum\limits_{w\in W^y} \left \vert f^y_n(w)-g f^x_n(w)\right \vert = \left \Vert f^y_n-g f^x_n\right \Vert_{\ell^1(W^y)} \longrightarrow 0 \text{ as } n\rightarrow \infty
\end{align*}
Now, we find a bound for ~\cref{eq2}.
\begin{align*}
&\sum_{v\in V^y}\left |  \sum\limits_{w\in W^x} \left [ a_{\sigma_w,n}(\phi(g w)^{-1} v) -  a_{\sigma_w,n}(h_{g,w}^{-1}\phi(g w)^{-1} v)\right ] f^x_n(w) \right |\\
&\hspace{.15in}\leq \sum\limits_{v\in V^y}\sum\limits_{w\in W^{x}} \abs{a_{\sigma_w,n}(\phi(g w)^{-1} v) - h_{g,w} a_{\sigma_w,n}(\phi(g w)^{-1} v)} f^{x}_n(w)\\
&\hspace{.15in}=\sum_{w\in W^x \cap W_n} f_n^x(w) \left \Vert a_{\sigma_w,n}^{r(h_{g,w})}-h_{g,w} a_{\sigma_w,n}^{s(h_{g,w})}\right \Vert_{\ell^1(V^{r(h_{g,w})})}\leq 2^{-n}
\end{align*}
We show for $\mu^1$-almost every $g$, this last inequality holds when $n$ is large enough. Let $E_{i,k}:=D_i \setminus D_k'$. Notice that $\mu^{1}(E_{i,k})\leq 2^{-k}\mu^1(D_i)$ for $k\geq i$ and so $\sum_k \mu^1(E_{i,k})<\infty$. Thus, $\mu^1(\limsup_k E_{i,k})=0$ and, in fact, $\mu^1(\bigcup_{i} \limsup_k E_{i,k})=0$. By ~\cref{restriction}, the set $\bigcup_{i} \limsup_k E_{i,k}$ contains exactly the $g\in\G$ where the inequality fails to hold for infinitely many $n$. Thus, we get that item (2) in the definition of amenability is satisfied for the sequence $(\xi_n)$. 

It remains to check item (1). By the definition of $\xi_n$, we have
\[ \xi_n^x(v)=\sum\limits_{w\in W^x} f^x_n(w) [\phi(w)  a_{\sigma_w,n}^{s(\phi(w))}(v)].\]
The measure $\phi(w) a^{s(\phi(w))}_{\sigma_w,n}$ is a probability measure since it is a pushforward of a probability measure. The function $\xi_n^x$ is a convex combination of probability measures, and therefore, is a probability measure.
\end{proof}

\section{Cost of Inner Amenable Groupoids}

We generalize the following structural result about inner amenable groupoids from the setting of groups\cite[Theorem 8]{TD14}.

\begin{thm} \label{iaimpqnorm}
If $\G$ is an inner amenable groupoid, and $\h \leq \G$ is a nowhere amenable subgroupoid, then there is a groupoid $\K$ such that $\h$ is q-normal in $\K$ and $\K$ is q-normal in $\G$. Moreover, $\K$ can be chosen so that, for every $n\in \mathbb{N}$, the groupoid $\K\cap \K^{\sigma_1}\cap \dotsc \cap \K^{\sigma_n}$ is aperiodic for all bisections $\sigma_1,\dotsc, \sigma_n \in [\G]$.
\end{thm}

\begin{proof}
By Lusin-Novikov\cite{Ke95}, there is a countable subgroup $H\leq [\h]$ of the full group that generates $\h$. We can then define the action groupoid $\tilde\h:=H\ltimes \G^0$ which comes with a natural surjective groupoid homomorphism $\varphi: \tilde\h\rightarrow \h$ that satisfies  $s(\varphi(h,x))=x$ and $r(\varphi(h,x))=h x$.

We define two different actions of $\tilde\h$ on $\G$. The first being the action by conjugation $\alpha: \tilde\h \acts \G$ where $(h,r(g))\cdot g=\varphi(h,r(g))g\varphi(h,s(g))^{-1}$.  For clarity in the rest of this proof, we use $\cdot$ to denote the conjugation action $\alpha$. Fix a countable collection of measurable partial sections $\Sigma$ of the range map $r$ for the conjugation action such that $\tilde \h \cdot\sigma'\cap\tilde \h\cdot \sigma''=\varnothing$ for $\sigma'\neq\sigma''$ and $\G=\bigsqcup_\Sigma \tilde \h \cdot \sigma$. The second action will be by left translation $\lambda: \tilde\h \acts \G$ where $(h,r(g)) g=\varphi(h,r(g))g$. 

Let $A_\sigma \subseteq dom(\sigma)$ be the unique maximal (mod null) set such that $ \left(\h\cap \h^{\sigma}\right)_{A_\sigma}$ is nowhere amenable. Define $\Sigma_A=\left\{\sigma\cap s^{-1}(A_\sigma) \:|\: \sigma \in \Sigma\right\}$ and define $\K$ to be the groupoid generated by $\h\cup \Sigma_A$. Also, define $\Sigma_B=\{\sigma\setminus\K|\sigma\in\Sigma\}$ and note that $\G\setminus \K=\tilde\h\cdot \Sigma_B$ by the assumptions on $\Sigma$ and since $\K$ is invariant under the action of $\tilde \h$. If $\tau\in \Sigma_A$ then $\h\cap \h^\tau$ is nowhere amenable on $s(\tau)$ and, in particular, aperiodic on $s(\tau)$. And if $h\in H$, then it's immediate that $\h\cap \h^h=\h$ is aperiodic. So we get that $\h$ is q-normal in $\K$.

Since $\G$ is inner amenable, there is a mean $m$ on $\G$ as in ~\cref{iadef}. In particular, this mean is equidistributed with respect to the map $r$. We proceed by contradiction to show $m(\K)=1$.

Assume not. So $m(\G\setminus\K)>0$. The mean $m$ is no longer necessarily equidistributed with respect to $r:\G\setminus \K\rightarrow \G^0$. However, we can define a finite Borel measure $\mu_{\G\setminus\K}$ on $\G^0$ by $\mu_{\G\setminus \K}(A):=m(r^{-1}(A)\setminus \K)$ which is absolutely continuous with respect to $\mu^0$. 

\begin{claim}
$\mu_{\G\setminus \K}$ is countably additive. 
\end{claim}
\begin{proof}[Proof of claim]
The mean $m$ is assume to be equidistributed, so for any measurable partition $P=(A_1,\dots,A_k,\dots)$ of $\G^0$
\[1=\sum_k \mu^0(A_k)=\sum_k m(r^{-1}(A_k))\]
which implies that for measurable $D_k\subseteq r^{-1}(A_k)$
\[m(\bigsqcup_k D_k)=\sum_k m(D_k).\]
Letting $D_k=r^{-1}(A_k)\setminus \K$, we see that $\mu_{\G\setminus \K}$ is countably additive.
\end{proof}

This $\mu_{\G\setminus\K}$ and $\mu^0$ are both $\h$-conjugation invariant, so the Radon-Nikodym derivative $f:=\frac{d\mu_{\G\setminus \K}}{d\mu^0}$ is also $\h$-conjugation invariant. Additionally, $f$ is bounded above by $1$ and not almost everywhere equal to $0$ by our assumption that $m(\G\setminus\K)>0$. Therefore, we may find a small enough $\varepsilon>0$ such that the $\h$-conjugation invariant set $W^0:=\{x|f(x)>\varepsilon\}$ has $\mu^0$-positive measure. Let $W:=r^{-1}(W^0)\setminus \K$. We define a conjugation-invariant positive linear functional $m_W\in (L^\infty(r^{-1}(W^0),\mu^1))^*$ by $m_W(D):=\int_W \frac{1_D}{f}dm$. This is still in $(L^\infty)^*$ since $f$ is bounded above and below. Let $(f_k=\sum c_i^k 1_{A_i^k})$ be a non-decreasing sequence of simple functions converging to $\frac{1}{f}$ in $L^\infty$. The mean $m_W$ is now equidistributed with respect to the restricted map $r_W: W\rightarrow \G^0$, since for $A\subseteq W^0$,
\begin{align*}
    \mu^0(A)&=\int_A \frac{1}{f} d\mu_{\G\setminus\K}=\lim_k\int_A f_k d\mu_{\G\setminus \K}\\
    &=\lim_k\sum_i c_i^k\mu_{\G\setminus \K}(A_i^k\cap A)=\lim_k\sum_i c_i^k m(r_W^{-1}(A_i^k\cap A))\\
    &=\lim_k\sum_i c_i^k \int_W 1_{r_W^{-1}(A_i^k)\cap r_W^{-1}(A)} dm=\lim_k \int_W f_k 1_{r_W^{-1}(A)} dm\\
    &=\int_W \frac{1_{r_W^{-1}(A)}}{f} dm=m_W(r_W^{-1}(A))
\end{align*}
By $\h$-invariance of $\K$ and $W^0$, we have a restricted conjugation action of $\tilde \h\acts W$  which we refer to by $\tilde\alpha$. Now, define a new mean $\tilde m_W:=\frac{m_W}{m_W(W^0)}$ by renormalizing $m_W$. The mean $\tilde m_W$ satisfies the assumptions of ~\cref{eqamen} for the action $\tilde \alpha$ and, therefore, there exists a sequence $h_n$ witnessing amenability of $\tilde\alpha:\tilde \h_{W^0}\acts W$. 

The left translation action $\lambda:\tilde \h_{W^0} \acts W\subseteq \G$ is nonamenable since $\h$ is nowhere amenable. Define $\Sigma_W:=\{\tau\cap W|\tau\in \Sigma_B\}$. For $\tau \in \Sigma_W$, denote by $\tilde\h_\tau\subseteq \tilde \h_{W^0}$ the stabilizer of the section $\tau$ with respect to the action $\alpha$. Let $\lambda_{\tau} :\tilde\h_{\tau} \acts r^{-1}(r(\tau))$ be the action $\lambda$ restricted to $\tilde\h_\tau$.  If this action is nonamenable, then the action $\varphi(\tilde\h_\tau)\acts r^{-1}(r(\tau))$ by left translation is nonamenable. But then, the action $\varphi(\tilde\h_\tau)\acts \varphi(\tilde\h_\tau)$ by left translation is nonamenable. Indeed, otherwise we can use the amenability sequence from $\varphi(\tilde\h_\tau)$ to show the action $\varphi(\tilde\h_\tau)\acts r^{-1}(r(\tau))$ is amenable. So the groupoid $\varphi(\tilde\h_\tau)$ is nonamenable if the action $\lambda_{\tau}$ is nonamenable.


By ~\cref{lm1}, we know there exists a $\tau\in \Sigma_W$ such that $\varphi(\tilde\h_{\tau})$ is nonamenable. Since $\tilde\h_{\tau}$ stabilizes $\tau$ under conjugation, $\varphi(\tilde\h_{\tau})\subseteq \h^{\tau}$. So, $\h\cap \h^{\tau}\supseteq \varphi(\tilde\h_{\tau})$ is nonamenable. We may find a positive measure set $A_{\tau}$ such that $(\h\cap\h^{\tau})_{A_\tau}$ is nowhere amenable. Recall, $\tau=(\sigma\setminus s^{-1}(A_\sigma))\cap W$ for some $\sigma \in \Sigma$. The sets $A_\sigma$ and $A_\tau$ are disjoint. However, the groupoid $\left (\h\cap\h^{\sigma}\right)_{A_\sigma\cup A_{\tau}}$ is nowhere amenable which contradicts maximality of $A_\sigma$. Hence, $m(K)=1$.

The conjugated groupoid $\K^\sigma:=\left\{ k^\sigma \:|\: k\in \K\right\}$ will still have mean 1 since $m$ is conjugation invariant. So, $\K\cap \K^{\sigma_1}\cap \dotsc \cap \K^{\sigma_n}$ will have mean 1 in $\G$ for all bisections $\sigma_1,\dotsc, \sigma_n \in [\G]$. Since $m$ is diffuse and equidistributed, this means that $\K\cap \K^{\sigma_1}\cap \dotsc \cap \K^{\sigma_n}$ is aperiodic and hence $\K$ is q-normal in $\G$.
\end{proof}

Now we prove the main theorem.

\begin{thm}
Assume $\varphi: R\rightarrow \G$ is a principal groupoid extension of an inner amenable groupoid $\G$. Then $C_{\mu^0_R}(R)=1$.
\end{thm}

\begin{proof}
By looking at the ergodic decomposition of $\G$, it suffices to deal with the groupoid $\G$ being ergodic.
We prove this in two cases.

Assume first that the associated equivalence relation $R_\G:=\left \{(r(g),s(g))\:|\: g\in \G\right \}$ is finite. By ergodicity, $\G^0=\{x_1, \dots, x_n\}$. This means that by \cite{KiTD18}, the isotropy group $\G_{\{x_i\}}:=\{g\in \G\:|\:s(g=r(g)=x_i)\}$ is an inner amenable group. By \cite{TD14}, this has fixed price 1. Consider now the groupoid $\G':=\bigsqcup _{1\leq i \leq n} \G_{\{x_i\}}$. Any principal extension of $\G'$ is a union of exactly $n$ ergodic components each of which is a principal extension of some copy of $\G_{\{x_i\}}$ and so generated by a set of measure $1/n+\varepsilon$. So, $\G'$ has fixed price 1 and $C_{\mu^0_R}(\varphi^{-1}(\G'))=1$. By ~\cref{qnormext} and ~\cref{qnorm}, we get $C_{\mu^0_R}(R)=1$.

If the underlying equivalence relation $R_\G$ is instead aperiodic, fix $\varepsilon>0$ and define $\h'\leq \G$ as the maximal(mod null) ergodic amenable subgroupoid using Zorn's Lemma. The set of amenable ergodic subgroupoids is nonempty since $R_\G$ is aperiodic. Now, let $A\subseteq \G\setminus \h'$ with $0<\mu^1(A)<\varepsilon$. Define $\h:=\langle \h'\cup A\rangle$ so $\h$ is nonamenable. In particular, since $\h$ is ergodic nonamenable, it is nowhere amenable. Now, by ~\cref{iaimpqnorm} and ~\cref{qnormext}, there exists $\K$ such that $\varphi^{-1}(\h)$ is q-normal in $\varphi^{-1}(\K)$ and $\varphi^{-1}(\K)$ is q-normal in $R$.

Notice that $\varphi|_{\varphi^{-1}(\h')}$ is a principal groupoid extension of an amenable groupoid so $C_{\mu^0_R}(\varphi^{-1}(\h'))=1$.  Therefore, $C_{\mu^0_R}(\varphi^{-1}(\h))<1+\varepsilon$. 
Last, \cref{qnorm} gives us  
\(C_{\mu^0_R}(R)\leq C_{\mu^0_R}(\varphi^{-1}(\h))< 1+\varepsilon.\)
\end{proof}


\printbibliography

\end{document}